\documentclass[a4paper,10pt]{amsart}

\usepackage{fancybox}
\usepackage{amsmath}
\usepackage{amssymb}
\usepackage{amsthm}
\usepackage{mathrsfs}
\usepackage{epstopdf}
\usepackage[all]{xy}
\usepackage{epstopdf}
\usepackage{array}
\usepackage{amscd}
\usepackage{hyperref}

\usepackage{amsfonts}
\usepackage{amssymb}
\usepackage{mathrsfs}
\usepackage{tikz}

\DeclareFontFamily{U}{rsfs}{%
\skewchar\font127}
\DeclareFontShape{U}{rsfs}{m}{n}{%
<-6>rsfs5<6-8.5>rsfs7<8.5->rsfs10}{}
\DeclareSymbolFont{rsfs}{U}{rsfs}{m}{n}
\DeclareSymbolFontAlphabet
{\mathrsfs}{rsfs}
\DeclareRobustCommand*\rsfs{%
\@fontswitch\relax\mathrsfs}

\theoremstyle{plain}
\newtheorem{thm}{Theorem}[section]
\newtheorem{prop}[thm]{Proposition}
\newtheorem{lem}[thm]{Lemma}

\newtheorem{rmk}[thm]{Remark}

\newtheorem{prop-defi}[thm]{Proposition-Definition}
\newtheorem{thm-defi}[thm]{Theorem-Definition}
\newtheorem{lem-defi}[thm]{Lemma-Definition}
\newtheorem{cor-defi}[thm]{Corollary-Definition}

\newtheorem{conj}[thm]{Conjecture}

\newdimen\argwidth
\def\db[#1\db]{
 \setbox0=\hbox{$#1$}\argwidth=\wd0
 \setbox0=\hbox{$\left[\box0\right]$}
  \advance\argwidth by -\wd0
 \left[\kern.3\argwidth\box0 \kern.3\argwidth\right]}

\newcommand{\eE}{\mathcal{E}}

\newcommand{\hH}{\mathcal{H}}

\newcommand{\lL}{\mathcal{L}}

\newcommand{\oO}{\mathcal{O}}

\newcommand{\Hom}{\mathop{\rm Hom}\nolimits}

\newcommand{\dR}{\mathbf{R}}

\newcommand{\Pic}{\mathop{\rm Pic}\nolimits}

\newcommand{\ch}{\mathop{\rm ch}\nolimits}

\newcommand{\Ext}{\mathop{\rm Ext}\nolimits}

\newcommand{\Coh}{\mathop{\rm Coh}\nolimits}

\newcommand{\ev}{\mathop{\rm ev}\nolimits}

\newcommand{\cneq}{\mathrel{\raise.095ex\hbox{:}\mkern-4.2mu=}}
\newcommand{\eqcn}{\mathrel{=\mkern-4.5mu\raise.095ex\hbox{:}}}

\newcommand{\DT}{\mathop{\rm DT}\nolimits}

\newcommand{\RHom}{\mathop{\dR\mathrm{Hom}}\nolimits}

\begin{document}
\title{Counting conics on sextic 4-folds}
\author{Yalong Cao}
\address{Kavli Institute for the Physics and Mathematics of the Universe (WPI),The University of Tokyo Institutes for Advanced Study, The University of Tokyo, Kashiwa, Chiba 277-8583, Japan}
\email{yalong.cao@ipmu.jp}

\maketitle
\begin{abstract}
We study rational curves of degree two on a smooth sextic 4-fold and their counting invariant defined using Donaldson-Thomas theory of 
Calabi-Yau 4-folds. By comparing it with the corresponding Gromov-Witten invariant, we verify a conjectural relation between them proposed by the author, Maulik and Toda.
\end{abstract}


\section{Introduction}
Let $X$ be a smooth sextic 4-fold in $\mathbb{P}^{5}$ and $M_{\beta}$ be the moduli scheme of one dimensional stable sheaves on $X$ with Chern character $(0,0,0,\beta,1)$. 
We are interested in the counting invariant of $M_{\beta}$ defined using Donaldson-Thomas theory of Calabi-Yau 4-folds, introduced in \cite{BJ, CL}. In particular, there exists a virtual class
\begin{equation}[M_{\beta}]^{\mathrm{vir}}\in H_2(M_{\beta},\mathbb{Z}). \nonumber \end{equation}
And we may use insertions to define counting invariants: for a class $\gamma\in H^4(X, \mathbb{Z})$, let
\begin{equation}
\tau(\gamma):=\pi_{M\ast}(\pi_X^{\ast}\gamma \cup\ch_3(\eE))\in H^2(M_{\beta},\mathbb{Z}),  \nonumber \end{equation}
where $\pi_X$, $\pi_M$ are projections from $X \times M_{\beta}$
onto corresponding factors and $\ch_3(\eE)$ is the
Poincar\'e dual to the fundamental class of the universal sheaf $\mathcal{E}$. 

The degree matches and we define $\DT_4$ invariants as follows
\begin{align*}
\mathrm{DT}_{4}(\beta\textrm{ }| \textrm{ } \gamma):=
\int_{[M_{\beta}]^{\rm{vir}}}\tau(\gamma) \in\mathbb{Z}.
\end{align*}
Since the definition of the virtual class involves a choice of orientation on certain (real) line bundle over $M_{\beta}$, the invariant will also depend
on that (see Sect. 2.1 for more detail).

Another obvious way to enumerating curves on $X$ is by GW theory. More specifically, for $\gamma\in H^{4}(X,\mathbb{Z})$, 
the genus 0 Gromov-Witten invariant of $X$ is 
\begin{equation}
\mathrm{GW}_{0, \beta}(\gamma):=\int_{[\overline{M}_{0, 1}(X, \beta)]^{\rm{vir}}}\mathrm{ev}^{\ast}(\gamma)\in\mathbb{Q},
\nonumber \end{equation}
where $\mathrm{ev} \colon \overline{M}_{0, 1}(X, \beta)\to X$ is the evaluation map.

In a previous work \cite{CMT}, the author, Maulik and Toda proposed a conjectural relation between $\DT_4$ invariants for one dimensional 
stable sheaves on $X$ and genus zero $\mathrm{GW}$ invariants of $X$ (see Conjecture \ref{conj:GW/GV} for details). 
The main result of this note is to verify this conjecture for degree two curve class on a smooth sextic 4-fold.
\begin{thm}(Theorem \ref{GW/DT4})
Conjecture \ref{conj:GW/GV} is true for degree one and two classes of a smooth sextic 4-fold $X\subseteq \mathbb{P}^{5}$, i.e.
for the line class $l\in H_2(X,\mathbb{Z})$ and any $\gamma\in H^{4}(X)$, 
we have 
\begin{equation}\mathrm{GW}_{0, l}(\gamma)=\mathrm{DT}_{4}(l \textrm{ }| \textrm{ } \gamma), \nonumber \end{equation} 
\begin{equation}\mathrm{GW}_{0, 2l}(\gamma)=\mathrm{DT}_{4}(2l \textrm{ }| \textrm{ } \gamma)+
\frac{1}{4}\cdot\mathrm{DT}_{4}(l\textrm{ }| \textrm{ } \gamma),  \nonumber \end{equation} 
for certain choice of orientation in defining the RHS.
\end{thm}
The proof of the above result relies on the study of the Hilbert scheme of conics on $X$. By the deformation invariance of $\DT_4$ and GW invariants, we may assume $X$ to be a generic hypersurface in $\mathbb{P}^5$, so that $M_{2l}$ is smooth of expected dimension whose closed points are structure sheaves of smooth conics or pairs of distinct intersecting lines (i.e. there is no double line) 
(ref. Proposition \ref{not contain double line}, \ref{virtual class for one dim sheaves}). 
A parallel study of the moduli space of stable maps shows that it consists of two components, one 
corresponding to the embedding of smooth or `broken' conics, another one corresponding to the 
double cover from $\mathbb{P}^1$ to lines on $X$. From this, we can conclude the above result.

\section{Geometry of moduli spaces of conics on sextic 4-folds}
Let $X\subseteq\mathbb{P}^{5}$ be a smooth sextic 4-fold, i.e. smooth degree 6 hypersurface in $\mathbb{P}^{5}$. By the adjunction formula, $X$ is a smooth projective Calabi-Yau 4-fold \cite{Yau}. 
We are interested in the moduli space of conics (degree two curves) in $X$. To be precise, let $I_{1}(X,2)$ be the moduli scheme of ideal sheaves on $X$ with Chern character $(1,0,0,-2,-1)$, 
which is isomorphic to the Hilbert scheme of one dimensional closed subschemes in $X$ with Hilbert polynomial $\chi(t)=2t+1$.
The inclusion $X\subseteq\mathbb{P}^{5}$ induces a closed embedding 
\begin{equation}         
I_{1}(X,2)\hookrightarrow \mathrm{Hilb}(\mathbb{P}^{5},2t+1)=\{\text{subscheme}  \textbf{ } C\subseteq \mathbb{P}^{5} \textbf{ } \text{with}  \textbf{ } \text{Hilbert} \textbf{ } \text{polynomial} 
\textbf{ }  2t+1 \} 
\nonumber \end{equation}
into the Hilbert scheme $\mathrm{Hilb}(\mathbb{P}^{5},2t+1)$ of conics in $\mathbb{P}^{5}$.  
\begin{lem}\label{hilb scheme of conics}
Let $S$ be the tautological rank three subbundle of the trivial bundle over $\mathrm{Gr}(3,6)$. Then there exists an isomorphism
\begin{equation}\mathrm{Hilb}(\mathbb{P}^{5},2t+1)\cong\mathbb{P}(\mathrm{Sym}^{2}(S^{*})),  \nonumber \end{equation}
where $\mathbb{P}(\mathrm{Sym}^{2}(S^{*}))$ is a $\mathbb{P}^{5}$-bundle over $\mathrm{Gr}(3,6)$.

Let $\pi: \mathrm{Hilb}(\mathbb{P}^{5},2t+1)\rightarrow \mathrm{Gr}(3,6)$ be the natural projection. Then the rank 13 vector bundle 
$\mathcal{E}=\pi^{*}\mathrm{Sym}^{6}(S^{*})/(T\otimes\pi^{*}\mathrm{Sym}^{4}(S^{*}))$ has a section whose zero locus is isomorphic to the moduli space 
$I_{1}(X,2)$ of conics in $X$. Here $T$ is the tautological line bundle over $\mathbb{P}(\mathrm{Sym}^{2}(S^{*}))$.
\end{lem}
\begin{proof}
As any conic in $\mathbb{P}^{5}$ is contained in a unique plane $\mathbb{P}^{2}\subseteq \mathbb{P}^{5}$, so we have the desired isomorphism (see e.g.  \cite[Lemma 2.2.6]{Deland}).

The description of $I_{1}(X,2)$ comes from a similar one for quintic threefold (\cite[Theorem 3.1]{Katz}), which we briefly recall as follows. 
An equation for $X$ induces a section of $\mathrm{Sym}^{6}(S^{*})$ hence a section of $\pi^{*}\mathrm{Sym}^{6}(S^{*})$. At a conic $C\in \mathrm{Hilb}(\mathbb{P}^{5},2t+1)$, the section represents a plane sextic cut out by the plane supporting $C$. The conic lies in $X$ if and only if $X$ factors into $C$ and a quartic. Such quartics globalize to the vector bundle $T\otimes\pi^{*}\mathrm{Sym}^{4}(S^{*})$ which is a subbundle of $\pi^{*}\mathrm{Sym}^{6}(S^{*})$. We consider the quotient bundle $\mathcal{E}=\pi^{*}\mathrm{Sym}^{6}(S^{*})/(T\otimes\pi^{*}\mathrm{Sym}^{4}(S^{*}))$. Then $I_{1}(X,2)$ is exactly isomorphic to the zero locus of the section of $\mathcal{E}$ induced by the defining equation of $X\subseteq\mathbb{P}^{5}$.
\end{proof}
Let $P:=\mathbb{P}(H^{0}(\mathbb{P}^{5},\mathcal{O}_{\mathbb{P}^{5}}(6)))\cong \mathbb{P}^{461}$ be the projective space of sextics in $\mathbb{P}^{5}$. For a `generic'
choice of sextic $X$ in $P$ (i.e. generic sextics means all sextics except those parameterized by a proper subvariety of $P$), the moduli space of conics in $X$ is smooth. To prove that, we need the following lemma.
\begin{lem}\label{surjective res map}
For any conic $C$ in $\mathbb{P}^{5}$, the restriction map 
\begin{equation}H^{0}(\mathbb{P}^{5},\mathcal{O}_{\mathbb{P}^{5}}(6))\rightarrow H^{0}(C,\mathcal{O}_{C}(6))  \nonumber \end{equation}
is surjective.
\end{lem}
\begin{proof}
Given a conic $C$, we have an exact sequence 
\begin{equation}H^{0}(\mathbb{P}^{5},\mathcal{O}_{\mathbb{P}^{5}}(6))\rightarrow H^{0}(C,\mathcal{O}_{C}(6))\rightarrow H^{1}(\mathbb{P}^{5},I_{C}(6))\rightarrow 0,  \nonumber \end{equation}
where $I_{C}$ is the ideal sheaf of $C$ in $\mathbb{P}^{5}$. We only need to prove $H^{1}(\mathbb{P}^{5},I_{C}(6))=0$.

Let $\mathbb{P}^{2}\subseteq \mathbb{P}^{5}$ be the unique plane containing $C$. By diagram chasing, we have a short exact sequence
\begin{equation}0\rightarrow I_{\mathbb{P}^{2}/\mathbb{P}^{5}}\rightarrow I_C \rightarrow \mathcal{O}_{\mathbb{P}^{2}}(-2)\rightarrow 0,  
\nonumber \end{equation}
where $I_{\mathbb{P}^{2}/\mathbb{P}^{5}}$ is the ideal sheaf of $\mathbb{P}^{2}$ in $\mathbb{P}^{5}$. By taking
$\dR\Gamma(\mathbb{P}^{5},-)$ and its cohomology, we have  
\begin{equation}H^{1}(\mathbb{P}^{5},I_{\mathbb{P}^{2}/\mathbb{P}^{5}}(6))\rightarrow H^{1}(\mathbb{P}^{5},I_{C}(6))\rightarrow 0.  \nonumber \end{equation}
So we are left to show $H^{1}(\mathbb{P}^{5},I_{\mathbb{P}^{2}/\mathbb{P}^{5}}(6))=0$. Similarly, from a short exact sequence 
\begin{equation}0\rightarrow I_{\mathbb{P}^{3}/\mathbb{P}^{5}}\rightarrow I_{\mathbb{P}^{2}/\mathbb{P}^{5}} \rightarrow \mathcal{O}_{\mathbb{P}^{3}}(-1)\rightarrow 0,  \nonumber \end{equation}
we are further reduced to show $H^{1}(\mathbb{P}^{5},I_{\mathbb{P}^{3}/\mathbb{P}^{5}}(6))=0$. By repeating the argument, the claim is reduced
to the obvious vanishing
$H^{1}(\mathbb{P}^{5}, I_{\mathbb{P}^{4}/\mathbb{P}^{5}}(6))=0$.
\end{proof}
\begin{prop}\label{smooth of moduli}
For a generic sextic $X\subseteq\mathbb{P}^{5}$, the moduli space $I_{1}(X,2)$ of conics in $X$ is a smooth projective curve.  

In particular, the Euler class of vector bundle $\mathcal{E}\rightarrow \mathrm{Hilb}(\mathbb{P}^{5},2t+1)$ in Lemma \ref{hilb scheme of conics} satisfies
\begin{equation}\iota_{*}[I_{1}(X,2)]=\mathrm{PD}(e(\mathcal{E}))\in H_{2}(\mathrm{Hilb}(\mathbb{P}^{5},2t+1),\mathbb{Z}),  \nonumber \end{equation}
where $\iota_{*}:I_{1}(X,2)\hookrightarrow \mathrm{Hilb}(\mathbb{P}^{5},2t+1)$ is the closed embedding.
\end{prop}
\begin{proof}
Let $\mathcal{I}$ be the incidence variety    
\begin{equation}\mathcal{I}=\{(C,X)\in \mathrm{Hilb}(\mathbb{P}^{5},2t+1)\times P \textrm{ }| \textrm{ } C\subseteq X \}  \nonumber \end{equation}
with projections $\pi_1: \mathcal{I}\rightarrow \mathrm{Hilb}(\mathbb{P}^{5},2t+1)$, $\pi_2: \mathcal{I}\rightarrow P$.

Given a conic $[C]\in \mathrm{Hilb}(\mathbb{P}^{5},2t+1)$, $\pi_1^{-1}(C)$ is the set of all sextics containing $C$. By Lemma \ref{surjective res map},
$H^{0}(\mathbb{P}^{5},\mathcal{O}_{\mathbb{P}^{5}}(6))\rightarrow H^{0}(C,\mathcal{O}_{C}(6))$ is surjective, hence $\pi_1$ is smooth. 
Therefore, $\mathcal{I}$ is irreducible, smooth of dimension 462.

By the generic smoothness (ref. \cite[Corollary 10.7, pp. 272]{Hart}), there exists a non-empty open subset $U\subseteq P$ such that
$\pi_{2}: \pi_{2}^{-1}(U)\rightarrow U$ is a smooth morphism. Hence a generic fiber of $\pi_2$ (i.e. $I_{1}(X,2)$ for a generic $X\subseteq \mathbb{P}^{5}$) is smooth of dimension one.
\end{proof}
Next, we show that a generic sextic 4-fold contains at most a finite number of broken conics and no double lines.
\begin{prop}\label{not contain double line}
A generic sextic $X\subseteq\mathbb{P}^{5}$ contains at most a finite number of broken conics (i.e. pairs of distinct intersecting lines), 
and no `double' lines.
\end{prop}
\begin{proof}
Any conic in $\mathbb{P}^{5}$ is contained in a unique plane $\mathbb{P}^{2}\subseteq\mathbb{P}^{5}$. It is either a smooth conic, a pair of distinct intersecting lines, or a `double' line (see e.g. \cite[Lemma 2.2.6]{Deland}). By Proposition \ref{smooth of moduli}, the moduli space of conics in a generic sextic 4-fold is a smooth projective curve.

We first show generic sextics do not contain double lines. Let   
\begin{equation}I_2=\{2l\in \mathrm{Hilb}(\mathbb{P}^{5},2t+1)  \textrm{ }|\textrm{ } l \textrm{ }\textrm{is} \textrm{ }\textrm{line}  \}  \nonumber \end{equation}
be the 11 dimensional variety of `double' lines in $\mathbb{P}^{5}$, and 
\begin{equation}\mathcal{C}_2=\{(2l,F)\in I_2\times \mathbb{P}(H^{0}(\mathbb{P}^{5},\mathcal{O}_{\mathbb{P}^{5}}(6))) \textrm{ }|\textrm{ } 2l \subseteq F^{-1}(0) \}.  \nonumber \end{equation}
The subspace of sextics (inside $\mathbb{P}(H^{0}(\mathbb{P}^{5},\mathcal{O}_{\mathbb{P}^{5}}(6)))$) containing one such double line $2l$ has dimension $h^{0}(\mathbb{P}^{5},\mathcal{O}_{\mathbb{P}^{5}}(6)))-1-h^{0}(C,\mathcal{O}_C(6))=448$. Hence $\mathcal{C}_2$ has dimension $459$. Thus a generic sextic can not lie in the image of $\pi_2: \mathcal{C}_2\rightarrow P$ as $P$ has dimension 461.

Let
\begin{equation}I_1=\{(l_1,l_2) \textrm{ }|\textrm{ } l_1, l_2\in\mathbb{P}^{5}, l_1\cap l_2\neq \emptyset \}  \nonumber \end{equation}
be the 13 dimensional variety of pairs of distinct intersecting lines in $\mathbb{P}^{5}$, and 
\begin{equation}\mathcal{C}_1=\{(l_1,l_2,F)\in I_1\times \mathbb{P}(H^{0}(\mathbb{P}^{5},\mathcal{O}_{\mathbb{P}^{5}}(6))) \textrm{ }|\textrm{ } l_1\cup l_2\subseteq F^{-1}(0) \}  \nonumber \end{equation}
with projections $\pi_1: \mathcal{C}_1\rightarrow I_1$, $\pi_2: \mathcal{C}_1\rightarrow P$.

The subspace of sextics (inside $\mathbb{P}(H^{0}(\mathbb{P}^{5},\mathcal{O}_{\mathbb{P}^{5}}(6)))$) containing one such intersecting lines $(l_1,l_2)$ has dimension $h^{0}(\mathbb{P}^{5},\mathcal{O}_{\mathbb{P}^{5}}(6)))-1-h^{0}(C,\mathcal{O}_C(6))=448$. 
As in Proposition \ref{smooth of moduli}, $\mathcal{C}_1$ is irreducible, smooth of dimension $461$.
The generic smoothness theorem implies that for a generic sextic $F$, $\pi_{2}^{-1}([F])$ is smooth of dimension zero.
\end{proof}

\section{Counting conics on sextic 4-folds}
There are many ways to count conics on a sextic 4-fold $X$. For instance, we can use $\mathrm{DT_ 4}$ invariants count   
one dimensional stable sheaves supported on conics inside $X$, as well as Gromov-Witten invariants count stable maps to conics in $X$. In this section, we compare them and verify a conjectural relation \cite{CMT} between GW invariants and $\mathrm{DT_ 4}$ invariants 
for one dimensional stable sheaves in this setting.

\subsection{Review of $\mathrm{DT_ 4}$ invariants}\label{sec:review}
We first review the framework for $\mathrm{DT_ 4}$ invariants.
We fix an ample divisor $\omega$ on $X$
and take a cohomology class
$v \in H^{\ast}(X, \mathbb{Q})$.

The coarse moduli space $M_{\omega}(v)$
of $\omega$-Gieseker semistable sheaves
$E$ on $X$ with $\ch(E)=v$ exists as a projective scheme.
We always assume that
$M_{\omega}(v)$ is a fine moduli space, i.e.
any point $[E] \in M_{\omega}(v)$ is stable and
there is a universal family
\begin{align}\label{universal}
\eE \in \Coh(X \times M_{\omega}(v)).
\end{align}

In~\cite{BJ, CL}, under certain hypotheses,
the authors construct 
a $\mathrm{DT}_{4}$ virtual
class
\begin{align}\label{virtual}
[M_{\omega}(v)]^{\rm{vir}} \in H_{2-\chi(v, v)}(M_{\omega}(v), \mathbb{Z}), \end{align}
where $\chi(-,-)$ is the Euler pairing.
Notice that this class will not necessarily be algebraic.

Roughly speaking, in order to construct such a class, one chooses at
every point $[E]\in M_{\omega}(v)$, a half-dimensional real subspace
\begin{align*}\Ext_{+}^2(E, E)\subset \Ext^2(E, E)\end{align*}
of the usual obstruction space $\Ext^2(E, E)$, on which the quadratic form $Q$ defined by Serre duality is real and positive definite. 
Then one glues local Kuranishi-type models of form 
\begin{equation}\kappa_{+}=\pi_+\circ\kappa: \Ext^{1}(E,E)\to \Ext_{+}^{2}(E,E),  \nonumber \end{equation}
where $\kappa$ is a Kuranishi map of $M_{\omega}(v)$ at $E$ and $\pi_+$ is the projection 
according to the decomposition $\Ext^{2}(E,E)=\Ext_{+}^{2}(E,E)\oplus\sqrt{-1}\cdot\Ext_{+}^{2}(E,E)$.  

In \cite{CL}, local models are glued in three special cases: 
\begin{enumerate}
\item when $M_{\omega}(v)$ consists of locally free sheaves only; 
\item  when $M_{\omega}(v)$ is smooth; 
\item when $M_{\omega}(v)$ is a shifted cotangent bundle of a derived smooth scheme. 
\end{enumerate}
And the corresponding virtual classes are constructed using either gauge theory or algebro-geometric perfect obstruction theory.
The general gluing construction is due to Borisov-Joyce \cite{BJ}, based on Pantev-T\"{o}en-Vaqui\'{e}-Vezzosi's theory of shifted symplectic geometry \cite{PTVV} and Joyce's theory of derived $C^{\infty}$-geometry. The corresponding virtual class is constructed using Joyce's
$D$-manifold theory.

The moduli space considered in this note is smooth of expected dimension, so the virtual class (up to sign) is simply the usual fundamental class
of the moduli space (see Prop. \ref{virtual class for one dim sheaves}). 

${}$ \\
\textbf{On orientations}.
To construct the above virtual class (\ref{virtual}) with coefficients in $\mathbb{Z}$ (instead of $\mathbb{Z}_2$), we need an orientability result 
for $M_{\omega}(v)$, which is stated as follows.
Let  
\begin{equation}\label{det line bdl}
 \lL:=\mathrm{det}(\dR \hH om_{\pi_M}(\eE, \eE))
 \in \Pic(M_{\omega}(v)), \quad  
\pi_M \colon X \times M_{\omega}(v)\to M_{\omega}(v),
\end{equation}
be the determinant line bundle of $M_{\omega}(v)$, equipped with a symmetric pairing $Q$ induced by Serre duality.  An \textit{orientation} of 
$(\mathcal{L},Q)$ is a reduction of its structure group (from $O(1,\mathbb{C})$) to $SO(1, \mathbb{C})=\{1\}$; equivalently, 
we require a choice of square root of the isomorphism
\begin{equation}\label{Serre duali}Q: \lL\otimes \lL \to \oO_{M_{\omega}(v)}  \end{equation}
to construct virtual class (\ref{virtual}).
An existence result of orientations is proved in \cite[Theorem 2.2]{CL2} for CY 4-folds $X$ such that $\mathrm{Hol}(X)=SU(4)$ and $H^{\rm{odd}}(X,\mathbb{Z})=0$\,\footnote{For instance, smooth sextic 4-folds satisfy this assumption.}.
Notice that, if orientations exist, their choices form a torsor for $H^{0}(M_{\omega}(v),\mathbb{Z}_2)$.

\subsection{$\mathrm{DT_{4}}$ virtual class for stable sheaves supported on conics}
Fix $\beta \in H_2(X, \mathbb{Z})\cong H^6(X,\mathbb{Z})$ and
\begin{align*}
v=(0, 0, 0, \beta, 1) \in H^0(X) \oplus H^2(X) \oplus H^4(X) \oplus H^6(X)
\oplus H^8(X),
\end{align*}
we set
\begin{align*}
M_{\beta}=M_{\omega}(0, 0, 0, \beta, 1)
\end{align*}
to be the Gieseker moduli space of semi-stable sheaves with Chern character $v$.
\begin{rmk}\label{rmk on stability}
Note that $M_{\beta}$ is
the moduli space of one dimensional
sheaves $E$'s on $X$ satisfying the following:
for any $0\neq E' \subsetneq E$, we have $\chi(E')\leqslant 0$.
In particular, it is
independent of the choice of $\omega$ and is a fine moduli space.
\end{rmk}
Since $\chi(v, v)=0$ in this case, we have
\begin{align*}
[M_{\beta}]^{\rm{vir}} \in H_2(M_{\beta}, \mathbb{Z}).
\end{align*}
To define invariants, we need insertions:
for a class $\gamma\in H^4(X, \mathbb{Z})$, let
\begin{equation}\label{tau gamma}
\tau(\gamma):=\pi_{M\ast}(\pi_X^{\ast}\gamma \cup\ch_3(\eE) ),  \end{equation}
\begin{align}\label{DT4}
\mathrm{DT}_{4}(\beta\textrm{ }| \textrm{ } \gamma):=
\int_{[M_{\beta}]^{\rm{vir}}}\tau(\gamma) \in\mathbb{Z}.
\end{align}
Here $\pi_X$, $\pi_M$ are projections from $X \times M_{\beta}$
onto corresponding factors, and $\ch_3(\eE)$ is the
Poincar\'e dual to the fundamental class of the universal sheaf $\mathcal{E}$.

For degree two class in a smooth sextic 4-fold $X$, the moduli space and its $\mathrm{DT}_{4}$ virtual class can be described as follows.
\begin{prop}\label{virtual class for one dim sheaves}
Let $X\subseteq\mathbb{P}^{5}$ be a generic sextic 4-fold and $\beta=2l \in H_2(X,\mathbb{Z})$ be the degree two class.
Then the moduli space $M_{2l}$ of one dimensional stable sheaves on $X$ has an isomorphism
\begin{equation} M_{2l}\cong I_1(X,2)  \nonumber \end{equation}
to the Hilbert scheme of conics in $X$.

Furthermore, the $\mathrm{DT}_{4}$ virtual class
\begin{equation}[M_{2l}]^{\mathrm{vir}}=[M_{2l}] \in H_2(M_{2l},\mathbb{Z})  \nonumber \end{equation}
is the usual fundamental class for certain choice of orientation in defining the LHS.
\end{prop}
\begin{proof}
By Proposition \ref{not contain double line}, we may assume $X$ contains smooth and broken conics only.

For $E\in M_{2l}$, $\chi(E)=1$ implies $h^{0}(E)\geqslant1$, so there exists a nontrivial section $s:\mathcal{O}_X\to E$.
If the image $\mathrm{Im}(s)\subseteq E$ is a proper subsheaf, then $\chi(\mathrm{Im}(s))\leqslant0$ by the stability of $E$.
Note that $\mathrm{Im}(s)$ is the structure sheaf of some one dimensional subscheme whose fundamental class is $l$ or $2l$, so 
$\chi(\mathrm{Im}(s))\leqslant0$ can not happen. Thus $s$ is surjective and $E\cong \mathcal{O}_C$ for some smooth or broken conic in $X$.
Conversely, when $C$ is smooth, $\mathcal{O}_C$ is obviously stable. 
As for a broken conic $C=l_1+l_2$, to test the stability of $\mathcal{O}_C$, we take a saturated (i.e. $E_2$ is pure) extension
\begin{equation}0\to E_1\to \mathcal{O}_C \to E_2 \to 0, \nonumber \end{equation}
we may assume $\mathrm{supp}(E_i)=l_i$ ($i=1,2$) without loss of generality. Then $E_i\cong \mathcal{O}_{l_i}(a_i)$ for some 
$a_i\in\mathbb{Z}$. From the exact sequence
\begin{equation}0\to \mathcal{O}_{l_1}(-1)\to \mathcal{O}_C \to \mathcal{O}_{l_2} \to 0, \nonumber \end{equation}
we have 
\begin{equation}0\to \Hom(\mathcal{O}_{l_1},\mathcal{O}_{l_1}(-1))\to  \Hom(\mathcal{O}_{l_1},\mathcal{O}_C)\to  \Hom(\mathcal{O}_{l_1}, \mathcal{O}_{l_2})=0. \nonumber \end{equation}
Hence $\Hom(\mathcal{O}_{l_1},\mathcal{O}_C)=0$. So $a_1\leqslant -1$, which implies the stability of $\mathcal{O}_C$. 

So we have a bijective morphism
\begin{equation}\theta: I_1(X,2)\to M_{2l}, \quad I_C\mapsto \mathcal{O}_C. \nonumber \end{equation}
Next, we compare their deformation-obstruction theory. For a conic $C\subseteq X$, there is a distinguished triangle
\begin{equation} \mathcal{O}_C \to I_C[1] \to \mathcal{O}_X[1], \nonumber \end{equation}
which implies the diagram
\begin{align*}
\xymatrix{    &  \dR \Gamma(\mathcal{O}_X)[1] \ar@{=}[r]\ar[d] &
\dR \Gamma(\mathcal{O}_X)[1] \ar[d] \\
\RHom(I_C, \mathcal{O}_C) \ar[r] & \RHom(I_C, I_C)[1] \ar[r] \ar[d] &
\RHom(I_C, \mathcal{O}_X)[1] \ar[d] \\
&  \RHom(I_C, I_C)_0[1]  &  \RHom(\mathcal{O}_C, \mathcal{O}_X)[2], }\end{align*}
where the horizontal and vertical arrows are distinguished triangles.
By taking cones, we obtain a distinguished triangle
\begin{align}\label{pair:tri}
\RHom(I_C, \mathcal{O}_C) \to \RHom(I_C, I_C)_0[1] \to
\RHom(\mathcal{O}_C, \mathcal{O}_X)[2].
\end{align}
Combining with distinguished triangle
\begin{equation}\RHom (\mathcal{O}_C,\mathcal{O}_C)\rightarrow \RHom (\mathcal{O}_X,\mathcal{O}_C)\rightarrow 
\RHom (I_C,\mathcal{O}_C),  \nonumber \end{equation}
and $h^{i\geqslant1}(\oO_C)=0$, we obtain canonical isomorphisms
\begin{equation}\Ext^{1}(\mathcal{O}_C,\mathcal{O}_C)\cong \Ext^{1}(I_C,I_C),  \nonumber \end{equation}
\begin{equation}\Ext^{2}(\mathcal{O}_C,\mathcal{O}_C)\cong \Ext^{2}(I_C,I_C).  \nonumber \end{equation}
By Proposition \ref{smooth of moduli}, for a generic sextic 4-fold $X$, $I_1(X,2)$ is smooth of dimension one. 
So $\Ext^{1}(I_C,I_C)\cong \mathbb{C}$, and $\Ext^{2}(I_C,I_C)=0$ by Riemann-Roch formula. 
Thus $\theta$ is an isomorphism and $M_{2l}$ is smooth of expected dimension. So the $\mathrm{DT}_4$ virtual class of $M_{2l}$ is its 
usual fundamental class for a choice of orientation.
\end{proof}

\subsection{GW invariants and $\mathrm{GW/DT_4}$ conjecture}
As the virtual dimension of the moduli space $\overline{M}_{0, n}(X, \beta)$ of genus zero, $n$-pointed stable maps is $1+n$, we need insertions to define GW invariants. 
For $\gamma\in H^{4}(X,\mathbb{Z})$, the genus 0 Gromov-Witten invariant of $X$ is defined to be 
\begin{equation}
\mathrm{GW}_{0, \beta}(\gamma):=\int_{[\overline{M}_{0, 1}(X, \beta)]^{\rm{vir}}}\mathrm{ev}^{\ast}(\gamma)\in\mathbb{Q},
\nonumber \end{equation}
where $\mathrm{ev} \colon \overline{M}_{0, 1}(X, \beta)\to X$ is the evaluation map.

The following conjecture is proposed in \cite{CMT} as an interpretation of Klemm-Pandharipande's Gopakumar-Vafa type invariants \cite{KP} on CY 4-folds in terms of $\mathrm{DT}_{4}$ invariants of one dimensional stable sheaves.
\begin{conj}$($\cite[Conjecture 1.3]{CMT}$)$\label{conj:GW/GV}
We have the identity
\begin{align*}
\mathrm{GW}_{0, \beta}(\gamma)=
\sum_{k|\beta}\frac{1}{k^{2}}\cdot\mathrm{DT}_{4}(\beta/k \textrm{ }| \textrm{ } \gamma),
\end{align*}
for certain choice of orientation in defining the RHS.
\end{conj}
We verify this conjecture for a smooth sextic 4-fold $X$ and degree one and two classes in $H_2(X,\mathbb{Z})\cong\mathbb{Z}$.
\begin{thm}\label{GW/DT4}
Conjecture \ref{conj:GW/GV} is true for degree one and two classes of a smooth sextic 4-fold $X\subseteq \mathbb{P}^{5}$, i.e.
for the line class $l\in H_2(X,\mathbb{Z})$ and any $\gamma\in H^{4}(X)$, 
we have 
\begin{equation}\mathrm{GW}_{0, l}(\gamma)=\mathrm{DT}_{4}(l \textrm{ }| \textrm{ } \gamma), \nonumber \end{equation} 
\begin{equation}\mathrm{GW}_{0, 2l}(\gamma)=\mathrm{DT}_{4}(2l \textrm{ }| \textrm{ } \gamma)+
\frac{1}{4}\cdot\mathrm{DT}_{4}(l\textrm{ }| \textrm{ } \gamma),  \nonumber \end{equation} 
for certain choice of orientation in defining the RHS.
\end{thm}
\begin{proof}
We only prove for the degree two curve class as the proof for line class is a simpler version of the same approach. 
As all invariants involved are deformation invariant, we could assume the sextic 4-fold $X$ to be generic so that the space $I_1(X,2)$ of 
conics in $X$ is smooth of dimension one and consists of smooth conics and at most a finite number of broken lines 
(as in Proposition \ref{smooth of moduli}, \ref{not contain double line}). 

The moduli space $\overline{M}_{0, k}(X, 2l)$ of $k$-pointed stable maps is the disjoint union of two connected components 
$\overline{M}_{0, k}(X, 2l)_{\mathrm{emb}}$, $\overline{M}_{0, k}(X, 2l)_{\mathrm{cov}}$, which parametrizes the embedding of smooth or broken conics into $X$ and double cover from $\mathbb{P}^{1}$ to lines in $X$ respectively.

We have a forgetful map
\begin{equation}\phi:\overline{M}_{0, 1}(X, 2l)_{\mathrm{emb}} \rightarrow \overline{M}_{0, 0}(X, 2l)_{\mathrm{emb}}\cong I_1(X,2),  \nonumber \end{equation}
\begin{equation}\phi:(f:C\rightarrow X, \textrm{ }p\in C)\mapsto (f:C\rightarrow X)  \mapsto  I_{f(C)},  \nonumber \end{equation}
and a natural embedding 
\begin{equation}i=(\phi,\ev): \overline{M}_{0, 1}(X, 2l)_{\mathrm{emb}}\hookrightarrow  I_1(X,2)\times X,  \nonumber \end{equation} 
\begin{equation}i(f:C\rightarrow X, \textrm{ }p\in C)=(I_{f(C)},f(P)) , \nonumber \end{equation}
whose image is the universal curve $\mathcal{Z}\subseteq I_1(X, 2)\times X$. Note that $\mathcal{Z}$ is an irreducible variety 
of dimension 4, so the virtual class of $\overline{M}_{0, 1}(X, 2l)_{\mathrm{emb}}$ is its usual fundamental class.

For $\gamma\in H^{4}(X)$, we have 
\begin{eqnarray*}
\int_{[\overline{M}_{0, 1}(X, 2l)_{\mathrm{emb}}]^{\mathrm{vir}}}\ev^{*}(\gamma)&=&\int_{[\overline{M}_{0, 1}(X, 2l)_{\mathrm{emb}}]}i^{*}(1\cup\gamma) \\
&=& \int_{i_*[\overline{M}_{0, 1}(X, 2l)_{\mathrm{emb}}]}\gamma  \\ 
&=& \int_{[I_1(X, 2)\times X]}(\mathrm{PD}([\mathcal{Z}])\cup\gamma)  \\
&=& \mathrm{DT}_{4}(2l \textrm{ }| \textrm{ } \gamma), 
\end{eqnarray*}
where the last equality is by Proposition \ref{virtual class for one dim sheaves}.

As for component $\overline{M}_{0, 1}(X, 2l)_{\mathrm{cov}}$, it can be identified as 
$\overline{M}_{0, 1}(\mathbb{P}^{1}, 2l)\times\overline{M}_{0, 0}(X, l)$ by
\begin{equation}\label{ident}\overline{M}_{0, 1}(\mathbb{P}^{1}, 2l)\times\overline{M}_{0, 0}(X, l)\cong \overline{M}_{0, 1}(X, 2l)_{\mathrm{cov}},\end{equation}
\begin{equation}(t:\mathbb{P}^{1}\rightarrow\mathbb{P}^{1},\textrm{ }P\in\mathbb{P}^{1}; \textrm{ }f:\mathbb{P}^{1}\rightarrow X)\mapsto (f\circ t: \mathbb{P}^{1}\rightarrow X, \textrm{ }P\in\mathbb{P}^{1}).   \nonumber \end{equation}
Note that $\overline{M}_{0, 1}(\mathbb{P}^{1}, 2l)$ is smooth of dimension 3.
For a generic sextic 4-fold $X$, lines in $X$ have normal bundle $\mathcal{O}_{\mathbb{P}^{1}}(-1,-1,0)$ 
and the moduli space $\overline{M}_{0, 0}(X, l)\cong I_1(X,1)$ is isomorphic to the (smooth) Fano scheme of lines 
(see \cite[Thm. 4.3, pp. 266 and Ex. 4.5 pp. 269] {Kollar}), which is one dimensional.
The obstruction space of $\overline{M}_{0, 1}(X, 2l)_{\mathrm{cov}}$ at $g:\mathbb{P}^{1}\rightarrow X$ is 
\begin{equation}H^{1}(\mathbb{P}^{1},g^{*}TX)\cong H^{1}(\mathbb{P}^{1},g^{*}\mathcal{O}_{\mathbb{P}^{1}}(-1,-1,0))
\cong H^{1}(\mathbb{P}^{1},g^{*}\mathcal{O}_{\mathbb{P}^{1}}(-1,-1)).
\nonumber \end{equation}
As $g$ varies, $H^{1}(\mathbb{P}^{1},g^{*}\mathcal{O}_{\mathbb{P}^{1}}(-1,-1))$ forms a rank two bundle which is the pull-back of an `obstruction' bundle $Ob\rightarrow \overline{M}_{0, 1}(\mathbb{P}^{1}, 2l)$ (its fiber over $f:\mathbb{P}^{1}\to \mathbb{P}^{1}$ is 
$H^{1}(\mathbb{P}^{1},f^{*}\mathcal{O}_{\mathbb{P}^{1}}(-1,-1))$). Hence under the isomorphism (\ref{ident}), the virtual class satisfies
\begin{equation}[\overline{M}_{0, 1}(X, 2l)_{\mathrm{cov}}]^{\mathrm{vir}}=\mathrm{PD}(e(Ob))\otimes [\overline{M}_{0, 0}(X, l)]\in 
H_2(\overline{M}_{0, 1}(\mathbb{P}^{1}, 2l))\otimes H_2(\overline{M}_{0, 0}(X, l)). \nonumber \end{equation}
We define 
\begin{equation}\pi: \overline{M}_{0, 1}(\mathbb{P}^{1}, 2l)\times \overline{M}_{0, 0}(X, l)
\rightarrow \overline{M}_{0, 0}(\mathbb{P}^{1}, 2l)\times \overline{M}_{0, 0}(X, l)\times X,
\nonumber \end{equation}
\begin{equation}\pi\big(t:\mathbb{P}^{1}\rightarrow\mathbb{P}^{1},\textrm{ }P\in \mathbb{P}^{1},\textrm{ }f:\mathbb{P}^{1}\rightarrow X\big)=
\big(t:\mathbb{P}^{1}\rightarrow\mathbb{P}^{1},\textrm{ }f:\mathbb{P}^{1}\rightarrow X,\textrm{ }f\circ t(P)\big).   \nonumber \end{equation}
By base change (e.g. \cite[pp. 182]{CK}), the obstruction bundle $Ob\rightarrow\overline{M}_{0, 1}(\mathbb{P}^{1}, 2l)$ is the pullback of an obstruction bundle $Ob\rightarrow \overline{M}_{0, 0}(\mathbb{P}^{1}, 2l)$ via the forgetful map $F:\overline{M}_{0, 1}(\mathbb{P}^{1}, 2l)\rightarrow \overline{M}_{0, 0}(\mathbb{P}^{1}, 2l)$. 
Then for $\gamma\in H^{4}(X)$, we have 
\begin{eqnarray*}
\int_{[\overline{M}_{0, 1}(X, 2l)_{\mathrm{cov}}]^{\mathrm{vir}}}\ev^{*}(\gamma)&=&\int_{[\overline{M}_{0, 1}(\mathbb{P}^{1}, 2l)\times\overline{M}_{0, 0}(X, l)]}\pi^{*}\big(e(Ob)\cup\gamma\big) \\
&=& \int_{\pi_*[\overline{M}_{0, 1}(\mathbb{P}^{1}, 2l)\times\overline{M}_{0, 0}(X, l)]}e(Ob)\cup\gamma \\ 
&=& \bigg(\int_{F_*[\overline{M}_{0, 1}(\mathbb{P}^{1}, 2l)]}e\big(Ob\big)\bigg)\cdot\bigg(\int_{2[\mathcal{C}]}\gamma\bigg)  \\
&=& 2\int_{[\overline{M}_{0, 0}(\mathbb{P}^{1}, 2l)]}e\big(Ob\big)\cdot\int_{[\overline{M}_{0, 0}(X, l)\times X]}\big(\mathrm{PD}([\mathcal{C}]\big)\cup\gamma)  \\
&=& 2\cdot\frac{1}{2^{3}}\cdot\mathrm{DT}_{4}(l \textrm{ }| \textrm{ } \gamma), 
\end{eqnarray*}
where $\mathcal{C}\subseteq\overline{M}_{0, 0}(X, l)\times X$ is the universal line under the identification $\overline{M}_{0, 0}(X, l)\cong I_1(X,1)$, and
the last equality is by the Aspinwall-Morrison formula (e.g. \cite[Lemma 27.5.3, pp. 547]{HKKPTVVZ}) and identification of virtual classes
\begin{equation}[M_{l}]^{\mathrm{vir}}=[I_1(X,1)]^{\mathrm{vir}}=[I_1(X,1)]=[\overline{M}_{0, 0}(X, l)]^{\mathrm{vir}}=[\overline{M}_{0, 0}(X, l)], 
\nonumber \end{equation} 
which can be obtained by a similar argument as in Proposition \ref{virtual class for one dim sheaves}.

To sum up, we obtain 
\begin{eqnarray*}
\mathrm{GW}_{0, 2l}(\gamma)&=& \int_{[\overline{M}_{0, 1}(X, 2l)_{\mathrm{emb}}]^{\mathrm{vir}}}\ev^{*}(\gamma)+\int_{[\overline{M}_{0, 1}(X, 2l)_{\mathrm{cov}}]^{\mathrm{vir}}}\ev^{*}(\gamma)  \\
&=& \mathrm{DT}_{4}(2l \textrm{ }| \textrm{ } \gamma)+\frac{1}{4}\cdot\mathrm{DT}_{4}(l\textrm{ }| \textrm{ } \gamma),
\end{eqnarray*}
i.e. Conjecture \ref{conj:GW/GV} is true for degree two class.
\end{proof}
\begin{cor-defi}
Let $X\subseteq \mathbb{P}^{5}$ be a smooth sextic 4-fold and $H$ be its hyperplane class. 
Then the number of lines, conics incident to 4-cycle $H^{2}$ is  
\begin{equation}\mathrm{DT}_{4}(l \textrm{ }| \textrm{ } H^{2})=60480, \quad
\mathrm{DT}_{4}(2l \textrm{ }| \textrm{ } H^{2})=440884080, \nonumber \end{equation} 
for certain choice of orientation in defining the LHS.
\end{cor-defi}
\begin{proof}
By Theorem \ref{GW/DT4}, for certain choice of orientation, 
we have $\mathrm{DT}_{4}(kl\textrm{ }| \textrm{ } \gamma)=n_{0,kl}(\gamma)$ for $k=1,2$, where $n_{0,kl}(\gamma)$ are 
Klemm-Pandharipande's genus zero GV type invariants defined using multiple cover formula and GW invariants. $n_{0,kl}(\gamma)$ are computed
in \cite[Table 2, pp. 33]{KP} by Picard-Fuchs equations and mirror principle of Lian-Liu-Yau and Givental. 
\end{proof}

${}$ \\
\textbf{Acknowledgement.} The author is grateful to Yukinobu Toda for useful discussions.
The work is supported by The Royal Society Newton International Fellowship.


\begin{thebibliography}{99}
\bibitem{BJ} D. Borisov and D. Joyce, \textit{Virtual fundamental classes for moduli spaces of sheaves on Calabi-Yau four-folds},
Geom. Topol. (21), (2017) 3231-3311.
\bibitem{CL} Y. Cao and N. C. Leung, \textit{Donaldson-Thomas theory for Calabi-Yau 4-folds}, arXiv:1407.7659.
\bibitem{CL2} Y. Cao and N. C. Leung, \textit{Orientability for gauge theories on Calabi-Yau manifolds}, Adv. in Math. (314), 2017, 48-70.
\bibitem{CMT} Y. Cao, D. Maulik and Y. Toda, \textit{Genus zero Gopakumar-Vafa type invariants for Calabi-Yau 4-folds}, Adv. in Math. 338 (2018), 41-92.
\bibitem{CK} D. A. Cox and S. Katz, \textit{Mirror symmetry and algebraic geometry}. (English summary) 
Mathematical Surveys and Monographs, 68. American Mathematical Society, Providence, RI, 1999. 
\bibitem{Deland} M. F. DeLand, \textit{Geometry of Rational Curves on Algebraic Varieties}, Ph.D. thesis, Columbia University, 2009.
\bibitem{Hart} R. Hartshorne, \textit{Algebraic geometry}, Graduate Texts in Mathematics, No. 52. Springer-Verlag, New York-Heidelberg, 1977. 
\bibitem{HKKPTVVZ}K. Hori, S. Katz, A. Klemm, R. Pandharipande, R. Thomas, C. Vafa, R. Vakil and E. Zaslow,
\textit{Mirror symmetry} with a preface by Vafa. Clay Mathematics Monographs, 1. American Mathematical Society, Providence, RI; Clay Mathematics Institute, Cambridge, MA, 2003.
\bibitem{Katz} S. Katz, \textit{On the finiteness of rational curves on quintic threefolds}, Compositio Math. 60 (1986), no. 2, 151-€"162. 
\bibitem{KP} A. Klemm and R. Pandharipande, \textit{Enumerative geometry of Calabi-Yau 4-folds}, Comm. Math. Phys. 281, 621-653 (2008). 
\bibitem{Kollar} J. Koll\'{a}r, \textit{Rational curves on algebraic varieties}, Ergebnisse der Mathematik und ihrer Grenzgebiete 32,
Springer Verlag, Berlin, 1996. 
\bibitem{PTVV}T. Pantev, B. T\"{o}en, M. Vaqui\'{e} and G. Vezzosi, \textit{Shifted symplectic structures},
Publ. Math. I.H.E.S. 117 (2013), 271-328.
\bibitem{Yau} S. T. Yau, \textit{On the Ricci curvature of a compact K\"{a}hler manifold and the complex Monge--Amp\`{e}re equation. I}, Comm. Pure Appl. Math. {\bf 31} (1978), no. 3, 339--411.


\end{thebibliography}
\end{document}